\documentclass[12pt]{amsart}
\usepackage{amssymb,latexsym}
\theoremstyle{definition}

\newtheorem{theorem}{Theorem}
\newtheorem{corollary}[theorem]{Corollary}
\newtheorem{proposition}[theorem]{Proposition}
\newtheorem{lemma}[theorem]{Lemma}
\newtheorem{definition}[theorem]{Definition}
\newtheorem{example}[theorem]{Example}
\newtheorem{notation}[theorem]{Notation}
\newtheorem{remark}[theorem]{Remark}

\newcommand{\numberset}{\mathbb}

\newcommand{\R}{\numberset{R}}

\newcommand{\F}{\numberset{F}}

\newcommand{\Pro}{\numberset{P}}

\newcommand{\Ol}{\mathcal{O}}

\usepackage{hyperref}

\usepackage[margin=2.6cm]{geometry}

\usepackage{fouriernc}

\usepackage{amsaddr}

\begin{document}

\title[]{On the geometry of Hermitian one-point codes}

\author{Edoardo Ballico$^1$}
\address{Department of Mathematics, University of Trento\\Via Sommarive 14,
38123 Povo (TN), Italy}
\email{$^1$edoardo.ballico@unitn.it}

\author{Alberto Ravagnani$^2$$^*$}
\address{Institut de Math\'{e}matiques, Universit\'{e} de 
Neuch\^{a}tel\\Rue
Emile-Argand 11, CH-2000 Neuch\^{a}tel, Switzerland}
\email{$^2$alberto.ravagnani@unine.ch}

\thanks{$^1$Partially supported by MIUR and GNSAGA}
\thanks{$^*$Corresponding author}

\maketitle

\providecommand{\bysame}{\leavevmode\hbox to3em{\hrulefill}\thinspace}

\begin{abstract}
 Here we describe the Algebraic Geometry of one-point codes arising from the
Hermitian curve. In particular, a geometric characterization of the
minimum-weight codewords 
of 
their dual codes is provided, including explicit closed formulas for their
number. 
We discuss also some natural improvements of the duals of Hermitian 
one-point codes by means of geometric arguments. Finally, some cohomological tools are
developed to characterize the small-weight codewords of such codes.
 
\end{abstract}

\tableofcontents

\setcounter{section}{-1}

\section{Introduction}
The aim of this paper is to use algebraic-geometric techniques to describe the
dual codes
of one-point codes on the Hermitian curve. Classical tools of Algebraic Geometry
have been
recently shown to be of great interest also for coding-theoretic purposes (see,
in particular,
\cite{c3}). The paper by A. Couvreur provides a geometric way to
lower-bound the dual
minimum distance of a wide class of evaluation codes arising from geometric
constructions. 
Here we restrict to the case of curves in the projective 
plane, and improve the cited work for the special case of one-point codes
arising from
the Hermitian curve. These codes are probably the most studied
algebraic-geometric codes. Our improvement is essentially due to the
introduction of
zero-dimensional
subschemes of the plane in the framework of \cite{c3}. This choice allows us to
give
a cohomological interpretation to Hermitian one-point codes and study them by
means of classical geometric tools.

\subsection{Main references}
 The minimum distance of
Hermitian one-point codes was completely determined in \cite{ST}.
Section 8.3 of \cite{Sti} is devoted to the study of such codes,
and most of their properties can be found there. The interest in the dual
minimum-weight codewords of Hermitian one-point codes dates back to \cite{SMP},
whose results have been recently extended in \cite{MPS}. For codes arising from
higher-degree places on the Hermitian curve, see \cite{Kor}. Improvements of
Goppa codes arising from the Hermitian curve have been recently studied in
\cite{DK}. Efficient decoding (and list-decoding) algorithms for Hermitian
one-point codes are known and well-studied in \cite{LO1}, \cite{LO2} and
\cite{OS}. A self-contained reference for codes arising from Algebraic
Geometry is the book by H. Stichtenoth (\cite{Sti}), which treats the
topic from the point of view of function fields. See \cite{Ste} for a more
geometric approach. Finally, this paper was inspired by the powerful
algebraic-geometric techniques introduced by A. Couvreur in \cite{c3}.

\subsection{Layout of the paper} The paper is organized as follows.
In Section \ref{preli} we recall the basic definitions of Coding Theory. In 
particular, we introduce Goppa codes on projective curves and briefly describe
the properties we are interested in. The Hermitian curve is defined in Section \ref{intr},
where we 
summarize some well-known results about one-point codes on it and give them a
cohomological interpretation.
Section \ref{secgeo} collects some preliminary results about the 
intersections of the Hermitian curve with lines and conics in the projective
plane.
In Section \ref{dualmindist} we provide a geometric characterization of the
dual minimum distance of codes arising from plane smooth curves,
and describe in details the case Hermitian codes by using the particular geometry
of tangent lines to the Hermitian curve. In Section
\ref{secminimum} we 
study the supports of the minimum-weight codewords of the dual codes of
Hermitian
one-point codes, providing explicit formulas for their number. The enumeration process
combines tools of classical and finite projective geometry. The analysis includes
come computational examples.
Natural improvements of the dual codes of
Hermitian one-point
codes are discussed in Section \ref{impro}, showing that the geometric approach offers
a complete control on the parameters of such constructions. Some interesting results on 
the dual small-weight codewords
of Hermitian one-point codes are stated in Section \ref{secsmall}. The proofs,
rather technical, are 
given in the algebraic-geometric Appendix \ref{appx}.

\section{Codes and algebraic curves}\label{preli}
Let $q$ be a prime power and let $\F_q$ denote the finite field with $q$
elements.
 A $q$-ary \textbf{linear code}
of dimension $k$ and length $n$ is simply a $k$-dimensional subspace of
$\F_q^n$.
We omit the adjective ``linear'' for the rest of the paper. The elements of
a code are called \textbf{codewords}, or simply \textbf{words}. Endow $\F_q^n$
with a 
metric space
 structure by defining the \textbf{Hamming
distance} $d:\F_q^n \times \F_q^n \to \R$  as $d(v,w):=|\{ 1 \le i \le n : v_i
\neq w_i \}|$, 
for any
$v=(v_1,...,v_n), w=(w_1,...,w_n) \in \F_q^n$.
By definition, the \textbf{weight} of a vector $v \in \F_q^n$ is its distance
from the $0$ vector,
 i.e. $\mbox{wt}(v):=d(v,0)=
|\{ 1 \le i \le n : v_i \neq 0 \}|$. The \textbf{minimum distance} of a code
$C \subseteq \F_q^n$ of at least two codewords is defined by
$$d(C):=\min_{v\neq w \in C} d(v,w) = \min_{v \in C \setminus \{ 0\}} \mbox{wt}
(v).$$ 
The minimum
distance of $\{0\}$ is taken to be, by definition, $\infty$.
The correction capability $c(C)$ of a code $C\subseteq \F_q^n$ is strictly
related to its minimum
 distance
 through the formula $c(C)=\lfloor (d(C)-1)/2\rfloor$ (for a minimum distance
decoder). 
For this reason, codes whose
minimum distance is high are very interesting for applications. 
The \textbf{weight distribution} of
a given code $C \subseteq \F_q^n$ is the collection ${\{ A_i(C)\}}_{i=0}^{n}$,
where $A_i(C)$ is 
given
by $A_i(C):=|\{ c \in C: \mbox{wt}(c)=i \}|$.
Define the \textbf{product} of $v=(v_1,...,v_n), w=(w_1,...,w_n) \in \F_q^n$ by 
$v \cdot w:=\sum_{i=1}^n v_i\cdot w_i$. The \textbf{dual code}
of a code $C \subseteq \F_q^n$ is denoted and defined by 
$C^\perp:=\{ v \in \F_q^n : v \cdot c=0 \mbox{ for any } c \in C\}$.
The set $C^\perp$ is a linear subspace of $\F_q^n$ of dimension
$n-\dim C$, i.e., a $q$-ary code of
dimension $n-\dim C$.

\begin{definition}
 Codes $C,D \subseteq \F_q^n$ are said to be \textbf{strongly isometric} if
$C=vD$, where
$v \in \F_q^n$ is a vector of non-zero components and 
$$vD:=\{ (v_1d_1,v_2d_2,...,v_nd_n) : (d_1,d_2,...,d_n) \in D \}.$$
\end{definition}
\begin{remark} \label{rrrrr}
 A strong isometry is an equivalence relation of codes. Strongly isometric
codes have the same
minimum distance and the same weight distribution. Two codes are strongly
isometric
if and only if their dual codes are strongly isometric. A strong isometry
preserves the supports 
of the codewords. In Geometric Coding Theory, studying codes
 up to strong isometries is a well-established praxis (see \cite{mp} for 
details).
\end{remark}

Here we recall the definition of Goppa code and state some well-known properties of codes
arising from projective curves.

\begin{definition}\label{Goppacode}
 Let $q$ be a prime power and let $\Pro^k$ be the projective space of dimension
$k$ over the field
 $\F_q$. Consider a smooth curve $X\subseteq \Pro^k$ defined over $\F_q$ and a
divisor $D$ on it. 
Take points $P_1,...,P_n \in X(\F_q)$ not lying in the support of $D$ and set
 $\overline{D}:=\sum_{i=1}^n P_i$.   
The \textbf{Goppa code} $C(\overline{D},D)$ is defined as the code obtained
evaluating the Riemann-Roch space 
space $L(D)$ at the points $P_1,...,P_n$.
\end{definition}
The construction of Definition \ref{Goppacode} was proposed in 1981
 by the Russian mathematician V. Goppa.
For a geometric introduction to Goppa codes see \cite{Ste} and \cite {Sti}.

\begin{definition}\label{spoint}
 Take the setup of Definition \ref{Goppacode}. Choose $s$ distinct
$\F_q$-rational points
of $X$, say $P_1,...,P_s$, and $s$ integers $a_1,...,a_s$.
Set $D:=\sum_{i=1}^s P_i$ and $\overline{D}:=\sum_{P \in X(\F_{q}) \setminus
\mbox{\footnotesize{Supp}}
 (D)} P$.
The code $C(\overline{D},D)$ is said to be an \textbf{$s$-point} code on $X$.
\end{definition}

Goppa codes are known to have good parameters and a designed minimum distance
(the so-called
Goppa bound, see \cite{Ste} again). From the true definition we see that curves
carrying
many rational points may give very interesting codes through Goppa's
construction, because of 
their length. 
A very useful fact in geometric Coding Theory is that linear equivalent divisors
give rise to strongly isometric Goppa codes.
\begin{remark}\label{ct}
 Take the setup of Definition \ref{Goppacode}. Let $D$ and $D'$ be divisors on
$X$ and take points 
$P_1,...,P_n \in X(\F_q)$ 
which do not appear neither in the support of $D$, nor in the support of $D'$.
 Set $\overline{D}:=\sum_{i=0}^n P_i$. It is known (see \cite{mp}, Remark 2.16)
that if
$D$ and $D'$ are linear equivalent divisors
 then  ${C}(\overline{D},D)$ and ${C}(\overline{D},D')$ are strongly isometric
codes.
In particular, ${C}(\overline{D},D)^\bot$ and ${C}(\overline{D},D')^\bot$ are
strongly isometric codes (see Remark \ref{rrrrr}).
\end{remark}

\section{Hermitian one-point codes} \label{intr}

Let $q$ be a prime power and let $\Pro^2$ denote the projective plane over the
field $\F_{q^2}$. 
Let $X \subseteq \Pro^2$ be the Hermitian curve (see \cite{Sti}, Example VI.3.6)
of 
affine equation $$y^q+y=x^{q+1}.$$ It is well-known that $X$ is a maximal curve
carrying 
$q^3+1$ $\F_{q^2}$-rational points (see for instance \cite{Ste}). Let $P_\infty$
be the only 
point at infinity of $X$, of projective coordinates $(0:1:0)$. 

\begin{notation} \label{cm}
 Let $m>0$ be an integer. We denote by $C_m$ be code obtained evaluating the
Riemann-Roch space 
$L(mP_\infty)$ on $B:=X(\F_{q^2})\setminus\{ P_\infty\}$. By Definition
\ref{spoint}, $C_m$ is
a one-point code on the Hermitian curve $X$, i.e., a \textbf{Hermitian one-point
code}.
\end{notation}
 Let $C_m$ be as in Notation \ref{cm}. It is well-known that $C_m^\bot$
(the dual code of $C_m$), is $C_{m_\bot}$, where $m_\bot$ is defined by
$m_\bot:=q^3+q^2-q-2-m$ 
(see \cite{Sti}, Theorem 2.2.8). The minimum distance of such codes has been
completely 
determined in \cite{ST}. Table \ref{tabsti} gives explicit formulas for 
the minimum distance of any non-trivial code $C_m$.

\begin{table}[h!!]
\centering

\caption{Minimum distance of any non-trivial code $C_m$.}

\scriptsize{
\begin{tabular}[h!]{|p{0.9cm}|p{3.5cm}|p{5.8cm}|}
\hline
Phase & Values of $m$ & Minimum distance  \\
\hline
\hline

\ \newline \ \newline 1 & \ \newline $0 < m < q^2-q$ \newline $m=\alpha q+\beta$
\newline $0 \leq \beta <q$ \newline \ &
\ \newline $q^3-\alpha(q+1)$, \mbox{ if $m<q$ or $m \ge q$ and $\alpha\le
\beta$}\newline \ \newline  $q^3-\beta-\alpha q$,  \mbox{ if $m \ge q$ and
$\alpha >\beta$}\newline \  \\

\hline

 \ \newline 2 \newline  & \ \newline $q^2-q \leq m < q^3-q^2$ \newline & \
\newline $q^3-m$ \newline \\

\hline

\ \newline \ \newline 3 \newline & \ \newline $q^3-q^2 \leq m < q^3$ \newline
$m=q^3-q^2+aq+b$ \newline $0 \leq a < b \leq q-1$ \newline & \ \newline \newline
$q^3-m$ \newline \\

\hline

\ \newline  \ \newline 4 & \ \newline $q^3-q^2 \leq m < q^3$ \newline
$m=q^3-q^2+aq+b$ \newline $0 \leq b \leq a \leq q-1$ \newline & \ \newline
\newline $q^3-m+b$ \newline \\

\hline
\ \newline \ \newline 5 &\ \newline $q^3 \leq m \leq q^3+q^2-q-2$ \newline
$m_\perp=\alpha q+\beta$ \newline $0 \leq \beta <q$ \newline \  & \ \newline
$\alpha +2$, \mbox{ if $m_\perp <q$ or $m_\perp \ge q$ and $\alpha\le
\beta$}\newline \ \newline  $\alpha +1$,  \mbox{ if $m_\perp \ge q$ and $\alpha
>\beta$}\newline \  \\

\hline

\end{tabular}
}

\label{tabsti}
\end{table}

As pointed out in the Introduction, the results of this paper are based upon
a geometric interpretation of Hermitian one-point codes. Our point of view is
explained in the following important note.

\begin{remark}\label{coho}
Let $m>0$ be an integer. A basis of the Riemann-Roch space $L(mP_\infty)$ is
given by the 
monomials
$$\{ x^iy^j \ : \ 0 \le j \le q-1 , \ iq+j(q+1) \le m \}$$
(see \cite{Sti}, proof of Proposition 8.3.2). 
Since $X$ is a maximal curve, for any $P \in X(\F_{q^2})$ we have an
isomorphism 
of sheaves $\Ol_X(1) \cong \mathcal{L}((q+1)P)$, the latter one being the
invertible 
sheaf associated to the divisor $(q+1)P$ on $X$. 
In other words, and more concretely, for any $d>0$ the Riemann-Roch space 
$L(d(q+1)P_\infty)$ is exactly the vector space $H^0(X,\Ol_X(d))$ of all the
degree
$d$ homogeneous forms defined on $X$. 
Moreover, for any integer $m>0$ there exists a 
unique pair of integers $(d,a)$ such that $m=d(q+1)-a$ and $0 \le a \le q$.
In this case we clearly have $d>0$. Write a linear equivalence $mP_\infty \sim
d(q+1)P_\infty-aP_\infty$. Set 
$E:=aP_\infty$ and denote by $H^0(X,\Ol_X(d)(-E))$ the vector space of all the
degree $d$
homogeneous forms defined on $X$ and vanishing on the scheme $E$ (this notation
is the standard one of algebraic geometers). We see
that $C_m$ is strongly isometric to the code obtained evaluating 
 $H^0(X,\Ol_{X}(d)(-E))$ on $B$. We denote this code by $C(d,a)$.
\end{remark}

The cohomological structure uderlined in Remark \ref{coho} will be studied in depth
in Section \ref{dualmindist}, after having summarized the main properties of the Hermitian curve.

\section{Projective geometry of the Hermitian curve}
\label{secgeo}
In this section we collect some results describing the intersections of
the Hermitian curve $X\subseteq \Pro^2$ with lines and conics in the plane.

\begin{lemma}
\label{intrette}
Let $X$ be the Hermitian curve. Every line $L$ of $\Pro^2$ either intersects $X$
in $q+1$ distinct $\F_{q^2}$-rational points, or $L$ is tangent to $X$ at a
point $P$ (with contact order $q+1$). In the latter case $L$ does not intersect
$X$ in any other $\F_{q^2}$-rational point different from $P$.
\end{lemma}
\begin{proof}
See \cite{pj}, part (i) of Lemma 7.3.2, at page 247.
\end{proof}

\begin{lemma}
\label{intcurve}
Let $X$ be the Hermitian curve. Fix an integer $e\in \{2,\dots
,q+1\}$ and a rational point $P\in 
 X(\mathbb {F}_{q^2})$. Let $E\subseteq X$ be the divisor $eP$, seen as a closed
degree $e$ subscheme of $\mathbb {P}^2$. Denote by $L_{X,P}\subseteq \mathbb
{P}^2$ be the tangent line
 to $X$ at $P$. It $T\subseteq \mathbb {P}^2$ is an effective divisor
 (i.e., a plane curve, possibly with multiple components)
 of degree $\le e-1$ and containing $E$, then $L_{X,P}\subseteq T$. In other
words, $L_{X,P}$ is one of the components of $T$.
\end{lemma}

\begin{proof}
Since $L_{X,P}$ has order of contact $q+1\ge e$ with $X$ at $P$, we have
$E\subseteq L_{X,P}$.
Since $\deg (E)>\deg (T)$ and $E\subseteq T\cap L_{X,P}$, Bezout theorem implies
$L_{X,P}\subseteq T$.
\end{proof}

The rational intersections of the Hermitian curve with some parabolas in the
affine chart $\{ z \neq 0\}$
of the projective plane $\Pro^2$ have been recently characterized in \cite{MPS}.

\begin{remark} \label{rempar}
 The authors of \cite{MPS} take affine parabolas of the form
$y=ax^2+bx+c$,
with $a,b,c \in \F_{q^2}$ and $a \neq 0$. 
A generic conic in $\Pro^2$ is given by a homogeneous equation of the
form
$$Ax^2+By^2+Cz^2+Dxy+Exz+Fyz=0,$$
with $A,B,C,D,E,F \in \F_{q^2}$ not all-zero. We see that $P_\infty$ is a
rational point of such conic if and only if $B=0$. Moreover, it is easily checked that
the tangent line to $X$ at $P_\infty$ has
equation $z=0$. This line is also the tangent line to the conic at $P_\infty$ if and
only if
$D=0$ and $F \neq 0$. It follows that
the parabolas studied in
\cite{MPS} are exactly the smooth conics in $\Pro^2$ passing through $P_\infty$
and tangent 
to $X$ at $P_\infty$.
\end{remark}

Thanks to the previous Remark \ref{rempar}, we can restate part of
\cite{MPS}, Theorem 3.1, in the following convenient form.

\begin{lemma}\label{paras}
Let $h>0$ be an integer and let  $\mathcal{T}(h)$ denote the set of the smooth
conics $T$
 in $\Pro^2$ passing through $P_\infty$, tangent to the Hermitian curve
$X$ at $P_\infty$, and satisfying  $\sharp(T \cap X \cap \{ z \neq 0\})=h$.
\begin{enumerate}
 \item Assume $q$ odd. Then $|\mathcal{T}(2q)|=q^2(q+1)(q-1)/2$. Moreover, if
$h>2q$ then $\mathcal{T}(h)=\emptyset$.
\item Assume $q$ even. If $h>2q-1$ then $\mathcal{T}(h)=\emptyset$.
\end{enumerate}
\end{lemma}

The results of this section will be applied throughout the rest of the paper in order to describe
the minimum distance and the minimum-weight codewords of the duals of Hermitian one-point codes.

\section{The dual minimum distance} \label{dualmindist}

The aim of this section is to give a geometric interpretation to the dual
minimum distance of certain codes arising from plane smooth curves. The results
improve the powerful method by A. Couvreur (see \cite{c3}) in the planar case,
and explicitly characterize the supports of minimum-weight codewords in terms
of cohomological vanishing conditions. The particular case of Hermitian
one-point codes is studied in depth.

\begin{proposition}\label{e1}
Let $\F$ be any field and let $\Pro^2$ denote the projective plane on $\F$.
Let $X \subseteq \Pro^2$ be a smooth plane curve. Fix an integer $d>0$, a
zero-dimensional
scheme $E\subseteq X$ and a finite
subset $B\subseteq X$ such that $B\cap E_{red}=\emptyset$\footnote{Here
$E_{red}$ denotes the reduction of the scheme $E$.}. Denote by $C$ the code 
obtained evaluating the vector space $H^0(C,\Ol_X(d)(-E))$ at the points of $B$.
Set $c:= \deg (X)$ and assume $d<c$. The following facts hold.

\begin{enumerate}
 \item The minimum distance
of $C^\perp$ is the minimal cardinality, say $z$, of a subset of $S \subseteq B$
such that $h^1(\mathbb {P}^2,\mathcal {I}_{S\cup E}(d)) >h^1(\mathbb
{P}^2,\mathcal
{I}_E(d))$.

\item A codeword of $C^\perp$
has weight $z$ if and only if it is supported by a subset $S\subseteq B$ such
that\footnote{We denote the cardinality of a finite set, say
$S$, by $\sharp(S)$.}:
\begin{enumerate}
 \item $\sharp{S} = z$,
\item $h^1(\mathbb {P}^2,\mathcal {I}_{E\cup S}(d)) >h^1(\mathbb
{P}^2,\mathcal {I}_E(d))$,
\item $h^1(\mathbb {P}^2,\mathcal {I}_{E\cup S}(d))
>h^1(\mathbb {P}^2,\mathcal {I}_{E\cup S'}(d))$ for any $S'\varsubsetneq S$.
\end{enumerate}
\end{enumerate}
\end{proposition}

\begin{proof}
Since $X$ is projectively normal (being a smooth plane curve) and we assumed
$d<c$, the restriction map
$\rho _d: H^0(\mathbb {P}^2,\mathcal {O}_{\mathbb {P}^2}(d)) \to H^0(X,\mathcal
{O}_X(d))$ is bijective.
Hence the restriction map $\rho _{d,E}: H^0(\mathbb {P}^2,\mathcal {I}_E(d)) \to
H^0(X,\mathcal {O}_X(d)(-E))$ is bijective. It follows that a finite subset
$S\subseteq
C \setminus E_{red}$
imposes independent condition to $H^0(X,\mathcal {O}_X(d)(-E))$ if and only
if $S$ imposes independent conditions to $H^0(\mathbb {P}^2,\mathcal {I}_E(d))$.
Moreover,
the set $S$ imposes independent conditions to $H^0(\mathbb {P}^2,\mathcal
{I}_E(d))$
if and only if $h^1(\mathbb {P}^2,\mathcal {I}_{E\cup S}(d)) =h^1(\mathbb
{P}^2,\mathcal {I}_E(d))$ (here we use again that $S\cap E_{red}=\emptyset$).
 To get the existence of a non-zero
codeword of $C^\perp$ whose support is $S$ (and not only with support contained
in $S$) we need that
the submatrix $M_S$ of the generator matrix of $C$ obtained by considering the
coloumns associated to the points appearing in $S$ has
the property
that each of its submatrices obtained deleting one coloumn have the same rank of
$M_S$ (each
such coloumn
is associated to some $P\in S$ and we require that the codeword has support
containing  $P$). This is equivalent to the last claim in the statement.
\end{proof}

\begin{remark}\label{possiamopossiamo}
 From now on, we explicitly focus on the analysis of one-point codes on the Hermitian
curve. As pointed out in Remark \ref{coho}, we need to study codes of type $C(d,a)$
with $d>0$ and $0 \le a \le q$.
In the
next lemma we show that we may restrict to the analysis of $C(d,a)$ codes such
that $d>0$ and $0 \le a \le d$.
\end{remark}

\begin{lemma}
\label{reduction}
 Let $X \subseteq \Pro^2$ be the Hermitian curve. Consider a  $C(d,a)$
code, with $d>1$ and $0 \le a \le q$. Set $E:=aP_\infty$.
\begin{itemize}
 \item If $a>d$ then set $d':=d-1$ and $a':=0$,
\item otherwise set $d':=d$ and $a':=a$.
\end{itemize}
  Then  
$C(d,a)$ and $C(d',a')$ are strongly isometric codes. In particular, their dual
codes are strongly ismoetric. 
\end{lemma}
\begin{proof} First of all, set $E':=a'P_\infty$.
The code $C(d,a)$ is obtained evaluating the degree $d$ homogeneous forms
vanishing
on the scheme $aP_\infty$. If
$a>d$ then an $f \in H^0(X,\Ol_X(d)(-aP_\infty))$ is divided by the equation of
the tangent line to $X$ at $P_\infty$, here denoted by $L_{X,P_\infty}$.
The division by such equation gives an isomorphism of vector spaces
$$H^0(X,\Ol_X(d)(-E)) \cong H^0(X,\Ol_X(d')(-E')).$$
Since $L_{X,P_\infty}$ does not intersect $X$ at any
rational point different
from $P_\infty$ (Lemma \ref{intrette}), we get that
$C(d,a)$ and $C(d',a')$ are in fact strongly isometric codes. Their dual codes
are also strongly isometric (Remark \ref{rrrrr}).
\end{proof}

The following lemma is rather technical and provides some cohomological
properties of zero-dimensional subschemes of $\Pro^2$. The result is taken from
\cite{br}, Lemma 7, and the proof is omitted here. 

\begin{lemma}\label{u00.01}
Let $\F_{q^2}$ be the finite field with $q^2$ elements ($q$ a prime power) and
denote 
by $\Pro^2$ the projective plane over the field $\F_{q^2}$. 
Let $X\subseteq \Pro^2$ be the Hermitian curve. Choose an integer $d>0$ and 
a zero-dimensional scheme $Z \subseteq X(\F_{q^2})$ of degree $z>0$. The
following facts hold.

\begin{itemize}
\item[(a)] If $z\le d+1$, then $h^1(\mathbb {P}^2,\mathcal {I}_Z(d))=0$.

\item[(b)] If $d+2 \le z\le 2d+1$, then $h^1(\mathbb {P}^2,\mathcal {I}_Z(d))>0$
if and only if there
exists a line $T_1$ such that $\deg (T_1\cap Z)\ge d+2$.

\item[(c)] If $2d+2\le z \le 3d-1$ and $d\ge 2$, then $h^1(\mathbb
{P}^2,\mathcal {I}_Z(d))>0$ if and only if
either there
exists a line $T_1$ defined over $\F_{q^2}$ such that $\deg (T_1\cap Z)\ge d+2$,
or there exists a conic $T_2$ defined over $\F_{q^2}$ such that $\deg (T_2\cap
Z) \ge 2d+2$.

\item[(d)] Assume $z=3d$ and $d\ge 3$. Then $h^1(\mathbb {P}^2,\mathcal
{I}_Z(d))>0$ if and only if
either there
exists a line $T_1$defined over $\F_{q^2}$ such that $\deg (T_1\cap Z)\ge d+2$,
or there is a conic $T_2$ defined over $\F_{q^2}$ such that $\deg (T_2\cap Z)
\ge 2d+2$, or there exists a plane cubic
$T_3$ such that $Z$ is the complete intersection of $T_3$ and a plane curve of
degree
$d$. In the latter case, if $d\ge 4$ then $T_3$ is unique and defined over
$\F_{q^2}$ and we may
find a plane curve $C_d$ defined over $\F_{q^2}$ and with $Z = T_3\cap C_d$.

\item[(e)] Assume $z \le 4d-5$ and $d\ge 4$.  Then $h^1(\mathbb {P}^2,\mathcal
{I}_Z(d))>0$ if and only if
either there
exists a line $T_1$ defined over $\F_{q^2}$ such that $\deg (T_1\cap Z)\ge d+2$,
or there exists a conic $T_2$ defined over $\F_{q^2}$ such that $\deg (T_2\cap
Z) \ge 2d+2$, or there exist $W\subseteq Z$ defined over $\F_{q^2}$ with $\deg
(W)=3d$ and plane cubic
$T_3$ defined over $\F_{q^2}$ such that $W$ is  the complete intersection of
$T_3$ and a plane curve of degree
$d$, or there is a plane cubic $C_3$ defined over $\F_{q^2}$ such that $\deg
(C_3\cap Z)\ge 3d+1$.
\end{itemize}
\end{lemma}

\begin{lemma} \label{h1=0}
 Let $X$ be the Hermitian curve. Choose integers $d>0$ and $0 \le a \le d$. Set
$E:=aP_\infty$. Then $h^1(\Pro^2,\mathcal{I}_E(d))=0$.
\end{lemma}

\begin{proof}
 Assume $h^1(\Pro^2,\mathcal{I}_E(d))>0$. Since $a \le d$, by Lemma
\ref{u00.01} there exists a line $L \subseteq \Pro^2$ such that $\deg(L\cap E)
\ge d+2$. Since in any case $\deg(L\cap E) \le d$ we immediately get a
contradiction.
\end{proof}

\section{Geometry of minimum-weight codewords}
\label{secminimum}

In this section an explicit description of the minimum-weight codewords of
$C(d,a)^\perp$ codes is provided, for any choice of $d$ with
$1 \le d \le q$. By Remark \ref{possiamopossiamo}, we restrict to the case $0
\le a \le d$. More precisely, we combine Proposition \ref{e1} with the other
preliminary results of Section \ref{secgeo} and Section \ref{dualmindist} in order to
geometrically characterize the supports of the minimum weight-codewords of
$C(d,a)^\perp$ codes. In particular, here we derive some explicit formulas  for
their number.

\begin{theorem} \label{minoreq}
 Let $d \le q-1$ be a positive integer. Take any integer $0 \le a \le d$ and
denote by $\delta:=\delta(d,a)$ the minimum distance of $C(d,a)^\perp$. Let
$A_\delta$ be the number of the minimum-weight codewords of $C(d,a)^\perp$.
\begin{enumerate}
 \item If $a=0$ then $\delta=d+2$ and a subset $S=\{ P_1,...,P_\delta
\} \subseteq X(\F_{q^2})\setminus \{ P_\infty\}$ of cardinality $\delta$ is the support of a
minimum-weight codewords of $C(d,a)^\perp$ if and only if it consists of
$\delta$ collinear points. Moreover,
\begin{equation*}
 \frac{A_\delta}{q^2-1}=
\left\{
\begin{array}{ll}
 \displaystyle (q^4-q^3)& \mbox{ if $d=q-1$}, \\
 \displaystyle  q^2 \binom{q}{\delta}+(q^4-q^3) \binom{q+1}{\delta} 
 & \mbox{ if $d<q-1$.}
\end{array}
\right.\
\end{equation*}

\item If $a>0$ then $\delta=d+1$ and a subset $S=\{ P_1,...,P_\delta
\}  \subseteq X(\F_{q^2})\setminus
\{ P_\infty\}$ of cardinality $\delta$ is the support
of a minimum-weight codeword of $C(d,a)^\perp$ if and only if $P_\infty,
P_1,...,P_\delta$ are
collinear points. Moreover, 
\begin{equation*}
 \frac{A_\delta}{q^2-1}=q^2\binom{q}{\delta}.
\end{equation*}
\end{enumerate}
\end{theorem}

\begin{proof}
 The minimum distance, $\delta(d,a)$, can be easily computed by
reversing Table \ref{tabsti}. Here we set $E:=aP_\infty$. Since $a\le d$, Lemma
\ref{h1=0} gives $h^1(\Pro^2, \mathcal{I}_E(d))=0$. By Proposition \ref{e1},
a subset 
$S=\{ P_1,...,P_\delta
\} \subseteq X(\F_{q^2})\setminus \{ P_\infty\}$ of cardinality $\delta$ is the support of a
minimum-weight codeword of $C(d,a)^\perp$ if and only if $h^1(\Pro^2,
\mathcal{I}_{E\cup S}(d))>0$. Let $S$ be with this property. 
\begin{enumerate}
\item Assume $a=0$. Then we have $E=\emptyset$ and $\deg(E)
+\sharp(S)=d+2 \le 2d+1$. Hence, by Lemma \ref{u00.01}, $h^1(\Pro^2,
\mathcal{I}_{E \cup S}(d))>0$  if and only if $P_1,...,P_\delta$ are collinear
points. By Lemma \ref{u00.01}, this condition is also sufficient for $S$ to be the support
of a minimim-weight codeword.
\item If $a>0$ then $\deg(E)
+\sharp(S)\le 2d+1$ and so, again by Lemma \ref{u00.01}, $h^1(\Pro^2,
\mathcal{I}_{E\cup S}(d))>0$  if and only if there exists a
line $L \subseteq \Pro^2$ such that $\deg(L \cap (E \cup S)) \ge d+2$. Since
$a \le d$ and Lemma \ref{intrette} holds, $L$ cannot be the tangent line to the
Hermitian curve $X$ at $P_\infty$. As a consequence, $P_\infty$ appears in $L$
with multiplicity one, and so $P_\infty, P_1,...,P_\delta$ are collinear points.
By Lemma \ref{u00.01}, this condition is also sufficient for $S$ to be the support
of a minimim-weight codeword.
\end{enumerate}
To get the formulas for the number of minimum-weight codewords, observe that, in
any linear, code two minimum-weight codewords with the same support are
(non-zero) multiple one each other. This fact follows from the definitions of linear code and 
minimum
distance. Moreover, any non-zero multiple of a minimum-weight codeword is an
other minimum-weight codeword with the same support. Hence we deduce our
formulas by using the properties of lines (Lemma \ref{intrette}).
\end{proof}

Theorem \ref{minoreq} describes the dual code of any one-point code $C_m^\perp$
with $m \le q^2-1$, providing explicit characterizations of the supports of its 
minimum-weight codewords. The following result, on the other hand, studies in
details $C(q,a)^\perp$ codes.

\begin{theorem} \label{te2}
 Set $d:=q$ and choose any integer $0 \le a \le d=q$.  Denote by
$\delta:=\delta(d,a)$ the minimum distance of $C(d,a)^\perp$.

\begin{enumerate}
 \item If $a=0$ then $\delta=2q+2$ and a subset $S=\{ P_1,...,P_\delta\}
\subseteq X(\F_{q^2}) \setminus \{ P_\infty\}$ of cardinality $\delta$ is the support of a
minimum-weight codeword of  $C(d,a)^\perp$ if and only if it is contained into
a conic of $\Pro^2$.

\item If $a=1$ then $\delta=2q+1$ and a subset $S=\{ P_1,...,P_\delta\}
\subseteq X(\F_{q^2}) \setminus \{ P_\infty\}$ of cardinality $\delta$ is the support of a
minimum-weight codeword of  $C(d,a)^\perp$ if and only if $P_\infty,
P_1,...,P_\delta$ lie on
a conic of $\Pro^2$.

\item If $2 \le a < q$ then $\delta=2q$. The following two facts hold.

\begin{enumerate}
\item Assume $q$ even. Then a subset $S=\{ P_1,...,P_\delta\}
\subseteq X(\F_{q^2}) \setminus \{ P_\infty\}$ of cardinality $\delta$  is the support of a
minimum-weight codeword of  $C(d,a)^\perp$ if and only if it is contained into
two lines meeting at $P_\infty$. Moreover,
\begin{equation*}
 A_\delta= (q^2-1) \binom{q^2}{2}.
\end{equation*}

\item Assume $q$ odd. Then a subset $S=\{ P_1,...,P_\delta\}
\subseteq X(\F_{q^2}) \setminus \{ P_\infty\}$ of cardinality $\delta$  is the support of a
minimum-weight codeword of  $C(d,a)^\perp$ if and only if either it is contained
into
two lines meeting at $P_\infty$, or it is contained into a smooth conic of
$\Pro^2$ which is tangent to $X$ at $P_\infty$. Moreover, 
\begin{equation*}
 A_\delta= (q^2-1) \left[q^2(q+1)(q-1)/2 + \binom{q^2}{2} \right].
\end{equation*}
\end{enumerate}
\end{enumerate}
\end{theorem}

\begin{proof}
 The dual minimum distance, $\delta=\delta(d,a)$, can be easily computed 
by reversing Table \ref{tabsti} at the beginning of the paper:
$$\delta(d,a)= \left\{ \begin{array}{ll} 2d+2-a & \mbox{ if $a \in \{ 0,1\}$,}
\\ 2d & \mbox{ if $a \ge 2$}.  \end{array}\right.\ $$
Set $E:=aP_\infty$. 
By Proposition \ref{e1} and Lemma \ref{h1=0}, a subset $S=\{
P_1,...,P_\delta\}
\subseteq X(\F_{q^2}) \setminus \{ P_\infty\}$ of cardinality $\delta$  is the support of a
minimum-weight codeword of $C(d,a)^\perp$ if and only if $h^1(\Pro^2,
\mathcal{I}_{E \cup S}(d))>0$. Let $S$ be with this property.

\begin{enumerate}
 \item If $a=0$, then $\delta=2d+2$ and
$\deg(E)+\sharp(S)=2d+2$. By Lemma \ref{u00.01}, there
exists either a
subscheme $W \subseteq S$ of degree $d+2=q+2$ and contained in a line, or a
subscheme $W \subseteq S$ of degree $2d+2=2q+2$ and contained in a conic. The
former case must be excluded because of Lemma \ref{intrette}. In the latter case
we see that $P_1,...,P_{2q+2}$ lie on a conic. By Lemma \ref{u00.01}, this
condition is necessary and sufficient for $S$ to the the support of a
minimum-weight codeword.

\item If $a=1$ then $\delta=2d+1$ and
$\deg(E)+\sharp(S)=2d+2$. By Lemma \ref{u00.01}, there exists either a
subscheme $W \subseteq P_\infty \cup S$ of degree $d+2=q+2$ and contained in a
line, or a subscheme $W \subseteq P_\infty \cup S$ of degree $2d+2=2q+2$ and
contained in a conic. The former case must be excluded because of Lemma
\ref{intrette}. In the latter case we have that $P_\infty, P_1,...,P_{2q+2}$ lie
on a conic. By Lemma \ref{u00.01}, this
condition is necessary and sufficient for $S$ to the the support of a
minimum-weight codeword.

\item Assume $a \ge 2$. We have $\delta=2d=2q$ and 
$\deg(E)+\sharp(S)=a+2d \le 3d-1$ (because we 
assumed $a<q=d$). Hence Lemma \ref{u00.01} applies: either there exists
a
subscheme $W \subseteq aP_\infty \cup S$ of degree $d+2=q+2$ and contained in a
line, or there exists a subscheme $W \subseteq aP_\infty \cup S$ of degree
$2d+2=2q+2$ and contained in a conic. The former case must be excluded, as in
the previous cases. If $W \subseteq aP_\infty \cup \{ P_1,...,P_{2d}\}$,
$\deg(W)=2d+2$ and $W$ is contained in a conic $T$ then the multiplicity of
$P_\infty$ in $W$, say $e_W(P_\infty)$, must be at least 2. On the other hand,
if $e_W(P_\infty)>2$ then (Lemma \ref{intcurve}) the tangent line to $X$ at
$P_\infty$, $L_{X,P_\infty}$, turns out to be a component of $T$. In this case
Lemma \ref{intrette} implies that $P_1,...,P_{2q}$ lie on the line
$T-L_{X,P_\infty}$, which contradicts Lemma \ref{intrette} again. As a
consequence, $e_W(P_\infty)=2$, and we are done. Indeed, $L_{X,P_\infty}$ cannot
be a component of $T$ (use Lemma \ref{intrette} twice) and so $T$ is either the
union of two lines meeting at 
$P_\infty$, or a smooth conic which is tangent to $X$ at $P_\infty$.
By Lemma \ref{u00.01}, this
condition is also sufficient for $S$ to the the support of a
minimum-weight codeword.

\begin{enumerate}
\item Assume $q$ even. By Lemma \ref{paras}, the case of the smooth conic must
be excluded. Hence $S$ is the support of a minimum-weight codeword of
$C(d,a)^\perp$ if and only if it is contained in the union of two plane lines
meeting at $P_\infty$. Since (Lemma \ref{intrette}) any line $L \subseteq
\Pro^2$ satisfies $\deg(L \cap X)=q+1$, and $\delta=2q$, the supports of the
minimum-weight codewords of
$C(d,a)^\perp$ are in bijection with the pairs of distinct lines of $\Pro^2$
passing through $P_\infty$ and not tangent to $X$ at $P_\infty$ (use Lemma
\ref{intrette} again).
The lines through $P_\infty$ are $q^2+1$. One of them is the tangent line to
$X$ at $P_\infty$. The formula follows.

\item If $q$ is odd, then we have to consider also the case of conics. By
Lemma \ref{paras}, a smooth conic $T \subseteq \Pro^2$ which is tangent to $X$
at $P_\infty$ cannot intersect $X$ in more than $2q$ affine points. Hence $S$
must appear exactly as the affine intersection of $X$ and such a conic. By Lemma
\ref{paras}, there exists $q^2(q+1)(q-1)/2$ conics with this property. Since
$2q>3$ and a parabola of the form $y=ax^2+bx+c$ (with $a,b,c \in \F_{q^2}$) is
completely determined by three of its points, distinct conics correspond
to distinct intersections. \qedhere
\end{enumerate}
\end{enumerate}
\end{proof}

\begin{remark}
 The formulas given in Theorem \ref{minoreq} and \ref{te2} extend those of
\cite{MPS}, proved for Hermitian one-point codes of minimum distance smaller or
equal than $q$.
\end{remark}

Theorem \ref{te2} concludes our analysis of the minimum-weight codewords of the
duals of Hermitian one-point codes. Let us examine two explicit
examples.

\begin{example}
 Take $q:=7$. The Hermitian curve is defined over $\F_{49}$ by the
affine equation
$y^7+y=x^8$. Here we study the dual of the one-point code on $X$ obtained evaluating the 
Riemann-Roch
space $L(53P_\infty)$ on the set $X(\F_{49})\setminus \{ P_\infty\}$. We notice that
such code require a significative computational effort, if studied by using a computer.
Write $53=7\cdot 8 -3$, so that $C_{53}$ is strongly isometric to the code $C(7,3)$.
Indeed, in the notation of Remark \ref{coho}, we have $d=q=7$ and $a=3$, with $0 \le a \le d$. 
Hence Theorem \ref{te2} applies. Since $3 \ge 2$, the minimum distance of
$C_{53}^\perp$ is $2q=14$. Moreover, the number of its minimum-weight codewords
is $66382848$. Their supports appear as the affine intersections of $X$ with plane conics
tangent to $X$ at $P_\infty$. 
\end{example}

\begin{example}
 Take $q:=9$. The Hermitian curve is defined over $\F_{81}$ by the
affine equation
$y^9+y=x^{10}$. Here we study the dual of the one-point code on $X$ obtained evaluating the 
Riemann-Roch
space $L(80P_\infty)$ on the set $X(\F_{81})\setminus \{ P_\infty\}$. As in the previous example,
write $80=8\cdot 10$, so that $C_{80}$ is strongly isometric to the code $C(8,0)$.
More precisely, in the notation of Remark \ref{coho}, we have $d=8=q-1$ and $a=0$. 
Hence Theorem \ref{minoreq} applies. Since $a=0$, the minimum distance of
$C_{80}^\perp$ is $d+2=q+1=11$. Moreover, the number of its minimum-weight codewords
is $466560$. Their supports appear as collinear points. 
\end{example}

\section{Improving Hermitian one-point codes}\label{impro}
In the previous sections we dealt with the duals of classical Hermitian
one-point codes $C_m$ obtained evaluating a Riemann-Roch space $L(mP_\infty)$ on
the set of points $B=X(\F_q^2)\setminus \{ P_\infty \}$, $X$ being the Hermitian
curve and $P_\infty$ its point at infinity. Now we discuss the possibility to
modify
the evaluation set $B$ in order to obtain codes which are shorter, but whose
minimum distance turns out to be improved. The geometric
interpretation of the dual minimum distance provided in Section \ref{e1}
offers a precise control of such improvements.
Other interesting improving constructions can be found in \cite{FR}. See also
the recent paper \cite{DK} by I. Duursma and R. Kirov for deep discussions on
the topic.

\begin{definition}\label{impr}
Fix any prime power $q$ and a positive integer $n$. Pick out a non-empty subset 
$H \subseteq \{ 1,...,n\}$ and denote by $\pi_H:\F_q^n \to \F_q^{n-\sharp(H)}$
the projection on the coordinates not-appearing in $H$. In other words, given a
vector
$v=(v_1,...,v_n) \in \F_q^n$, we delete the components associated to any index
$i \in H$ by
operating $\pi_H(v)$.
 Let $C \subseteq \F_q^n$ be a code and let $\delta$ be the minimum distance of
$C^\perp$. A subset
 $H\subseteq \{1,...,n\}$
is said to be an \textbf{improving subset} for $C^\perp$ if the minimum distance
of $\pi_H(C)^\perp$ is
strictly greater than $\delta$. The map $\pi_H$ will be called an
\textbf{improving projection} for the code $C^\perp$.
\end{definition}

\begin{definition}\label{conH}
 Let $d>0$ and $0 \le a \le q$ be integers. Set $B:=X(\F_{q^2})\setminus \{
P_\infty \}$ and choose a non-empty subset $H \subseteq B$.
We denote by $C(d,a,H)$ the code obtained evaluating the vector space
$H^0(X,\Ol_X(d)(-aP_\infty))$ on the set
$B \setminus H$ and by $C(d,a,H)^\perp$ its dual code.
\end{definition}

\begin{remark}
 By enumerating the points appearing in $B=X(\F_{q^2})\setminus \{ P_\infty \}$
we can identify $H$ with a subset of $\{ 1,...,n:=q^3\}$ and write
$C(d,a,H)=\pi_{H}(C(d,a))$ in the notations of Definition \ref{impr}.
\end{remark}

\begin{remark}
 The proof of Lemma \ref{reduction} still works if we replace $B$ and $C(d,a)$
with $B\setminus H$ and $C(d,a,H)$ (respectively). Hence, from now on, we will
consider only $C(d,a,H)$ codes with $d>0$ and $a \le d$.
\end{remark}

\begin{theorem}\label{carimp}
 Let $0<d<q$ and $1 \le a \le d$ be integers. Choose a non-empty subset $H
\subseteq B=X(\F_{q^2})\setminus \{P_\infty\}$
and let $C(d,a,H)$ be as in Definition \ref{conH}. The minimum distance of
$C(d,a,H)^\perp$ is at least $d+1$ and the equality holds if and only if there
exist $d+1$ collinear points in $B \setminus H$ on a line through $P_\infty$.
\end{theorem}

\begin{proof}
Since $1 \le a \le d$ we get, by setting $E:=aP_\infty$,
$h^1(\Pro^2,\mathcal{I}_E(d))=0$ (Lemma \ref{h1=0}). By Proposition \ref{e1} the
minimum distance of $C(d,a,H)^\perp$ is the smallest cardinality, say $\delta$,
of a subset $S \subseteq B \setminus H$ such that $h^1(\Pro^2,\mathcal{I}_{E\cup
S}(d))>0$.
Let $S:=\{ P_1,...,P_\delta\}$ be the support of a minimum-weight codeword of
$C(d,a,H)^\perp$. In particular, we have $h^1(\Pro^2,\mathcal{I}_{E \cup S}(d))>0$.
If $\delta \le d$ then $\deg(E \cup S) \le 2d$ and (Lemma \ref{u00.01})  there
exists a line $L \subseteq \Pro^2$ such that $\deg(L \cap (E \cup S)) \ge d+2$.
Since $E \cap S = \emptyset$ and $\deg(S) \le d$ we have that $P_\infty$ appears
in $L$ with multiplicity at least two. By Lemma \ref{intcurve} this means that
$L$ is the tangent line to $X$ at $P_\infty$ and (Lemma \ref{intrette}) $\deg(L
\cap (E \cup S))= \deg(E) \le d$, a contradiction. It follows $\delta \ge d+1$.
If $\delta=d+1$ then $\deg(E \cup S) \le 2d+1$ and so (Lemma \ref{u00.01}) there
exists a line $L \subseteq \Pro^2$ such that
$\deg(L \cap (E \cup S)) \ge d+2$. If $P_\infty$ appears in $L$ with
multiplicity greater than one then $L$ is tangent to $X$ at $P_\infty$,
contradicting Lemma \ref{intrette} ($\deg(E) \le d$ here). It follows that the
points $P_1,...,P_\delta$ lie on a line through $P_\infty$. Finally, assume that
$S=\{ P_1,...,P_\delta \}\subseteq B \setminus H$ is a set of $d+1$ collinear
points on a line through $P_\infty$ (this is possible because we assumed $d<q$,
and so $d+1<q+1$). By Lemma \ref{u00.01} we have
$h^1(\Pro^2,\mathcal{I}_{E \cup S}(d))>0$ and hence $S$ contains the support of
a minimum-weight codewords of $C(d,a,H)^\perp$. Since we proved that the minimum
distance of
$C(d,a,H)^\perp$ is at least $d+1$ we deduce that $S$ is in fact the support of
a minimum-weight codeword of $C(d,a,H)^\perp$. This concludes the proof.
\end{proof}

\begin{corollary}\label{CCC}
 Choose integers $0<d<q$ and $1 \le a \le d$. A non-empty subset $H \subseteq
B=X(\F_{q^2})\setminus \{ P_\infty \}$ is an improving subset for $C(d,a)^\perp$
if and only if there are no $d+1$ collinear points in $B \setminus H$ lying on a
line through $P_\infty$. In particular, if $H$ is an improving subset for
$C(d,a)^\perp$ then $\sharp(H) \ge q^2(q-d)$ and so the length of
$C(d,a,H)^\perp$ is at most $q^2d$. 
\end{corollary}

\begin{proof}
 By Theorem \ref{carimp}, to get an improving subset $H \subseteq B$ we must
remove from $B$ any $d+1$ points lying on a line through $P_\infty$. The lines
in $\Pro^2$ passing through $P_\infty$ and not tangent to $X$ are $q^2$. Every
such a line contains $q$ points different from $P_\infty$ (Lemma
\ref{intrette}). From any line we must remove at least $q-d$ points. This gives
the formula.
\end{proof}

\begin{remark}
 Corollary \ref{CCC} shows that, in order to improve the minimum distance of a
non-trivial Hermitian one-point code $C_m^\perp$ (with $m \le q^2-1$) by
evaluating the Riemann-Roch space $L(mP_\infty)$ on a proper subset of rational
points of the curve, it is necessary to reduce the length of the code in a
significative manner.
\end{remark}

\begin{example}
 Take $q:=5$. The Hermitian curve $X$ is defined over $\F_{25}$ by the affine equation 
$y^5+y=x^6$. Let us consider the Hermitian one-point code $C_{11}$. Its length is
$q^3=125$. In the notation
of Remark \ref{coho}, we have $d=2<q$ and $a=1< d$. Hence the dual minimum distance of 
$C_{4}$ is $3$ (see Theorem \ref{minoreq}). As in the statement of 
Corollary \ref{CCC}, take $H$ to ba a minimal improving subset for $C_{25}^\perp$.
The length of the improved code is $50$, while it was $125$.
\end{example}

\section{Geometry of small-weight codewords}
\label{secsmall}
 In this Section we state a result which describes the small-weight codewords of
certain
$C(d,a)^\perp$ codes. Our goal is to characterize the
supports of such
codewords from a geometric point of view.

\begin{remark} \label{abasso}
 By Lemma \ref{reduction}, for any $C(d,a)$ code with $d>1$ and $0 \le a \le q$
there exist integers $d'>0$ and $0 \le a' \le d'$ such that $C(d,a)=C(d',a')$.
Hence, from now on, we will consider only $C(d,a)$ codes with $d>0$ and $a \le
d$.
\end{remark}

The proof of Theorem \ref{finale} is
rather technical, and it needs some non-trivial algebraic-geometric
preliminaries. For these reasons, it is given in Appendix \ref{appx}. Notice
that the statement of the theorem can be perfectly understood without any
knowledge of Algebraic Geometry.

\begin{notation}
 We denote by $L_{X,P_\infty}$ the tangent line to the Hermitian curve $X$
at $P_\infty$. Moreover, $\mathcal{R}(\infty)$ will be the set of the lines
passing
through $P_\infty$ which are not tangent to $X$ in any point. $\mathcal{R}$ will
denote the set of the lines which do not contain $P_\infty$ and which are not
tangent to $X$ at any point.
\end{notation}

The following result provides a complete description of the small-weight
codewords of any $C(d,a)^\perp$ such that $d \le q-1$ and $0 \le a \le d$. By
Remark \ref{abasso},
here we describe the small-weight codewords of any non-trivial $C_m^\perp$ code
such that $m \le q^2-1$.

\begin{theorem}\label{finale}
 Let $0 < d \le q-1$ and $0 \le a \le d$ be integers. Denote by $S=\{
P_1,...,P_w\}$ be the support of a codeword of $C(d,a)^\perp$ of weight $w$.

\begin{enumerate}
 \item Assume $d+2 \le a+w \le 2d+1$. Then $S$ must be one of the
sets in the following list:
\begin{enumerate}
\item[(a)] a subset of $w$ elements of $L \cap B$, for an $L \in
\mathcal{R}(\infty)$ ($w \ge d+1$);
\item[(b)] a subset of $w$ elements of $L \cap B$, for an $L \in \mathcal{R}$
($w \ge d+2$).
\end{enumerate}
Moreover, any such a set appears as the support of a codeword of $C(d,a)^\perp$
of weight exactly $w$.

\item Assume  $2d+2 \le a+w \le 3d-1$. Then either $S$ is one of
the sets in cases (a), (b) of the previous list, 
\begin{enumerate}
 \item[(c)] or there exist two distinct lines $L,M \subseteq \Pro^2$ such that

\begin{itemize}
\item $\deg(L \cap (E \cup S)) \ge d+2$,
\item $\deg(M \cap (E \cup S)) \ge d+1$,
\item $\deg((L \cup M) \cap E)+w \ge 2d+2$,
\item either $w \ge 2d+3$ (if $L,M \in \mathcal{R}$), or $w \ge 2d+2$ (if 
$(L,M) \in \mathcal{R}\times \mathcal{R}(\infty)$ or $(M,L) \in
\mathcal{R}\times \mathcal{R}(\infty) $), or $w \ge 2d+1$ (if $L,M \in
\mathcal{R}(\infty)$),
\end{itemize}

\item[(d)] or there exists two distinct lines $L,M \subseteq \Pro^2$ such that
\begin{itemize}
\item $\deg(L \cap (E \cup S)) = \deg(M \cap (E \cup S)) = d+1$,
\item $\deg((L \cup M) \cap E)+w \ge 2d+2$,
\item $L \cap M \cap S = \emptyset$,
\item either $w=2d$ (if and only if $a \ge 2$ and $L \cap M=P_\infty$), or
$w=2d+1$
(if and only if $a \ge 1$ and $(L,M) \in \mathcal{R}\times \mathcal{R}(\infty)$,
or
$(M,L) \in \mathcal{R}\times \mathcal{R}(\infty) $), or $w=2d+2$ (if and only if
$L,M \in \mathcal{R}$),
\end{itemize}
\item[(e)] or there exists a smooth conic $T \subseteq \Pro^2$ such that
\begin{itemize}
\item $\deg(T \cap E)+w \ge 2d+2$,
\item $w \ge 2d+2- \min\{ 2,a\}$.
\end{itemize}
\end{enumerate}
\end{enumerate}
\end{theorem}

\begin{proof}
 See Appendix \ref{appx}.
\end{proof}

\begin{remark}
Notice that the number of the small-weight codewords of a $C(d,a)^\perp$
code cannot be derived here from the number of their supports, as in the proofs
of Theorem \ref{minoreq} and Theorem \ref{te2}. Indeed, two small-weight
codewords having the same support don't need to be proportional.
\end{remark}

\appendix

\section{Proof of Theorem \ref{finale}} \label{appx}

Here we prove Theorem \ref{finale}. We split the argument in some preliminary
lemmas of geometric content. We follow the notation of the rest of the paper.

\begin{lemma} \label{e2}
 Let $0<d<q+1$ and $0 \le a \le d$ be integers. Consider the Hermitian one-point
code $C(d,a)$. Set $B:=X(\F_{q^2})\setminus \{P_\infty \}$, $E:=aP_\infty$. Fix
a subset $S \subseteq B$ and an integer $e>0$. There exists a linear subspace of
$C(d,a)^\perp$ with support contained in $S$ if and only if
$h^1(\Pro^2,\mathcal{I}_{E\cup S}(d)) \ge e$.
\end{lemma}

\begin{proof}
 Set $V:= H^0(X,\mathcal {O}_X(d)(-E))$ and $V(-S):=
H^0(X,\mathcal {I}_{S\cup E}(d))$. Write $B = S\sqcup (B\setminus S)$ and
identify $K^S = \{S\to K\}$ with $K^S\times K^{B\setminus  S}$. The linear
projection
of $K^B$ onto its factor $K^S$ and the inclusion $V\hookrightarrow K^B$ induce
an inclusion $V/V(-S)\hookrightarrow \mathbb {F}_q^{B\setminus S}$. Fix $f\in
K^B$
with support on $S$. By the latter assumption we have $\sum _{P\in B} f(P)g(P) =
\sum _{P\in S} f(P)g(P)$ for all $g\in K^B$. The integer
$i(V,S):= \sharp (S) - h^0(X,\mathcal {O}_X(d)(-E))  +h^0(X,\mathcal
{O}_X(d)(-E-S))$
is the number of independent linear relations among the evaluations of $V$ at
the points of $S$.
Hence $i(V,B)$ is the dimension of the linear subspace of $C^{\bot}$ formed by
the words with support on $S$. 
As in the proof of Proposition \ref{e1} the restriction map
$\rho : H^0(\mathbb {P}^2,\mathcal {I}_E(d)) \to H^0(X,\mathcal {O}_X(-E))$ is
bijective.
Obviously $\mbox{Ker}(\rho )=H^0(\mathbb {P}^2,\mathcal {I}_X(d))$.
Since $i(V,S)$ is the number of conditions that $S$ imposes to $H^0(X,\mathcal
{O}_X(-E))$ and $S\subseteq X$,
we get $i(V,S) = h^0(\mathbb {P}^2,\mathcal {I}_E(d)) -h^0(\mathbb
{P}^2,\mathcal {I}_{E\cup S}(d))$.
Since $S\cap E =\emptyset$ and $h^1(\mathbb {P}^2,\mathcal {I}_E(d))=0$ (Lemma
\ref{h1=0}),
we have $i(V,S) = h^1(\mathbb {P}^2,\mathcal {I}_{S\cup E}(d))$.
\end{proof}

\begin{lemma}\label{e3}
Consider a code $C(d,a)$ with $0<d<q+1$ and $0 \le a \le d$. Set $E:=aP_\infty$.
For any integer $h$ such that
$1\le h \le \binom{d+2}{2} -\deg (E)$ the smallest minimum distance of a subcode
$C^\bot_h \subseteq C^\bot$ of dimension $h$  is the minimal cardinality of a
set $S\subseteq B$
such that $h^1(\mathbb {P}^2,\mathcal {I}_{S\cup E}(d)) \ge h$.
\end{lemma}

\begin{proof}
Apply Lemma \ref{e2}.
\end{proof}

\begin{remark}\label{e4}
Let $W$ be any projective scheme and $L$ a line bundle on it.  Fix any subscheme
$E\subseteq Z$. Since $Z$ is zero-dimensional we have $h^1(Z,\mathcal
{I}_{E,Z}\otimes L) >0$. Hence 
the restriction
map $H^0(Z,L\vert Z) \to H^0(E,L\vert E)$ is surjective. It follows that if
$h^1(W,\mathcal {I}_W\otimes
L) >0$ then $h^1(W,\mathcal {I}_Z\otimes L) >0$. 
\end{remark}

\begin{remark}\label{e5}
For any effective divisor $T\subseteq \mathbb {P}^2$ and any zero-dimensional
subscheme $Z\subseteq \mathbb {P}^2$
let $\mbox{Res}_T(Z)$ denote the residual scheme of $Z$ with respect to $T$,
i.e. the closed subscheme
of $\mathbb {P}^2$ with $\mathcal {I}_Z:\mathcal {I}_T$ as its ideal sheaf. We
have
$\deg (Z) = \deg (Z\cap T)+ \deg (\mbox{Res}_T(Z))$. If $Z = Z_1\sqcup Z_2$ then
$\mbox{Res}_T(Z) = \mbox{Res}_T(Z_1)\sqcup \mbox{Res}_T(Z_2)$. If $Z$ is reduced
(i.e. if $Z$ is a finite set)
then $\mbox{Res}_T(Z) = Z\setminus Z\cap T$. For each $d\in \mathbb {Z}$ we
have an exact sequence
\begin{equation}\label{eqb1}
0 \to \mathcal {I}_{\mbox{Res}_T(Z)}(d-k) \to \mathcal {I}_Z(d) \to \mathcal
{I}_{Z\cap T,T}(d)\to 0,
\end{equation}
where $k:= \deg (T)$. It follows that, for each integer $i\ge 0$, 
\begin{equation}\label{eqb2}
h^i(\mathbb {P}^2,\mathcal {I}_Z(d)) \le h^i(\mathbb {P}^2,\mathcal
{I}_{\mbox{Res}_T(Z)}(d-k))+h^i(T,\mathcal {I}_{Z\cap
T,T}(d)).
\end{equation}
\end{remark}

\begin{lemma}\label{e6}
Let $d>0$ be an integer and $T\subseteq \mathbb {P}^2$ be any divisor of degree
$k \le d+2$. Let $Z\subseteq T$ be any zero-dimensional scheme. Then
$h^1(\mathbb {P}^2,\mathcal {I}_Z(d))
=h^1(T,\mathcal {I}_{Z,T}(d))$. 
\end{lemma}

\begin{proof}
Since $Z\subseteq T$, we have $\mbox{Res}_T(Z) =\emptyset$. Hence the residual
exact sequence (\ref{eqb1}) becomes the
exact sequence
\begin{eqnarray*}
0 \to \mathcal {O}_{\mathbb {P}^2}(d-t) \to \mathcal {I}_Z(d) \to \mathcal
{I}_{Z,T}(d)\to 0.
\end{eqnarray*}
Use that $h^1(\mathbb {P}^2,\mathcal {O}_{\mathbb {P}^2}(d-k))=0$ and deduce
(since $d-k \ge -2$)
that $h^2(\mathbb {P}^2,\mathcal {O}_{\mathbb {P}^2}(d-k))=0$.
\end{proof}

\begin{lemma}\label{e7}
Let $0<d<q+1$ and $0 \le a \le d$ be integers. Set $E:=aP_\infty$ and fix any
line $L\subseteq \mathbb {P}^2$ and a set $S\subseteq L$. If $\sharp (S)-\sharp
(L\cap
S)+\deg (E)-\deg (E\cap L) \le d$, then 
$$h^1(\mathbb {P}^2,\mathcal {I}_{E\cup S}(d)) = h^1(L,\mathcal {I}_{(E\cup
S)\cap
L,L}(d)) = \max \{0,\deg (E\cap L)+\sharp (L\cap S)-d-1\}.$$
\end{lemma}

\begin{proof}
Since $E\cap S=\emptyset$, we have
\begin{enumerate}
 \item[(a)] $\deg (E\cup S) =\deg (E)+\deg (S)$,
\item[(b)] $\deg (\mbox{Res}_L(E\cup S))= \deg (\mbox{Res}_L(E)) + \sharp
(S)-\sharp (S\cap L)$,
\item[(c)] $\deg (L\cap (E\cup S))= \deg (E\cap L) +\sharp (S\cap L)$.
\end{enumerate}
The latter equality gives $$h^1(L,\mathcal {I}_{(E\cup S)\cap
L,L}(d))
= \max \{0,\deg (E\cap L)+\sharp (L\cap S)-d-1\},$$ because $L\cong \mathbb
{P}^1$. Since $\deg (\mbox{Res}_L(E\cup S))
\le d$, we have $h^1(\mathbb {P}^2,\mathcal {I}_{\mbox{Res}_L(E\cup L)}(d-1))=0$
(\cite{bgi}, Lemma 34, or \cite{ep}, Remarque (i) at p. 116). Hence equation
(\ref{eqb2}) leads to the inequality $h^1(\mathbb {P}^2,\mathcal {I}_{E\cup
S}(d)) \le h^1(L,\mathcal {I}_{(E\cup S)\cap
L,L}(d))$. Since $(E\cup S)\cap L \subseteq E\cup S$, Remark \ref{e4} and Lemma
\ref{e6} imply $h^1(\mathbb {P}^2,\mathcal {I}_{E\cup S}(d)) \ge h^1(L,\mathcal
{I}_{(E\cup S)\cap
L,L}(d))$.
\end{proof}

\begin{lemma}\label{c1}
Let $S\subseteq B$ be the support of a codeword of a code $C(d,a)^\bot$ with
$0<d<q+1$ and $0 \le a \le d$. Set $E:=aP_\infty$ and assume
the existence of a plane curve $T$ of degree $k$ such that $h^1(\mathbb
{P}^2,\mathcal {I}_{\mbox{Res}_T(E\cup S)}(d-k))=0$. Then $S\subseteq T$.
\end{lemma}

\begin{proof}
Let $V(S)$ (resp. $V(S\cap T)$) be the subcode of $C(d,a)^\bot$ formed
by the codewords whose support is contained in $S$ (resp. in $S\cap T$). We have
to prove that $V(S) = V(S\cap T)$. Obviously $V(S\cap T) \subseteq V(S)$. From
the sequence (\ref{eqb1}) we get $h^1(\mathbb {P}^2,\mathcal {I}_{E\cup S}(d)) =
h^1(\mathbb {P}^2,\mathcal {I}_{T\cap (E\cup S)}(d))$. Hence
Lemma \ref{e2} applied to $S\cap T$ and to $S$ gives  $V(S)\subseteq V(S\cap
T)$.
\end{proof}

\begin{proof}[Proof of Theorem \ref{finale}]
Let us divide our proof into several steps. 
\begin{enumerate}
\item Let $S \subseteq B$ be the support of a codeword of weight $w$ of
$C(d,a)^\perp$. Observe that $\sharp(S)=w$. By Proposition \ref{e1} we have
$h^1(\Pro^2,\mathcal{I}_{E\cup S}(d))>0$. Assume $d+2 \le a+w \le 2d+1$, i.e.
$\deg(E \cup S) \le 2d+1$. By Lemma \ref{u00.01} there exists a line $L
\subseteq \Pro^2$ (defined over $\overline{\F_{q^2}}$) such that $\deg(L \cap (E
\cup S)) \ge d+2$. Since $\deg (\mbox{Res}_L(E\cup S))\le 2d+1-d-2 \le d$, the
case $k= 1$ of Lemma \ref{c1} implies  $S\subseteq L$. Since $S \neq \emptyset$
and each point of $S$ is defined over $\F_{q^2}$ then also $L$ is defined over
$\F_{q^2}$. Set $W:=L \cap (E \cup S)$ and note that the multiplicity of
$P_\infty$ in $W$, say $e_W(P_\infty)$, must satisfies $e_W(P_\infty) \le 1$.
Indeed, if $e_W(P_\infty) \ge 2$ then Lemma \ref{intrette} implies
$L=L_{X,P_\infty}$, which contradicts $\deg(W) \ge d+2$ (we assumed $a \le d$).
Hence we have $\sharp(L \cap S) \ge d+1$ and the support $S$ consists  of $w$
points in $L\cap B$ for a 
certain $L \in \mathcal{R}(\infty) \sqcup \mathcal{R}$. On the other hand, let
$L \in \mathcal{R}(\infty) \sqcup \mathcal{R}$ and let $S \subseteq B \cap L$
with $\sharp(S)=w$. Assume $a+w \le 2d+1$. Observe that $\sharp(S)-\sharp(L \cap
S)+\deg(E)-\deg(E \cap L) \le w-w+a \le d$ and hence by Lemma \ref{e7} we have
$h^1(\Pro^2,\mathcal{I}_{E\cup S}(d))>h^1(\Pro^2,\mathcal{I}_{E\cup S'}(d))$ for
any $S' \varsubsetneq S$. Apply Proposition \ref{e1} and deduce that $S$ appears
as the support of a codeword of $C(d,a)^\perp$ of weight $w$.

\item Let $S \subseteq B$ be the support of a codeword of weight $w$ of
$C(d,a)^\perp$. Observe that $\sharp(S)=w$. By Proposition \ref{e1} we have
$h^1(\Pro^2,\mathcal{I}_{E\cup S}(d))>0$. Assume $2d+2 \le a+w \le 3d-1$. By
Lemma \ref{u00.01} there exists either a line $L \subseteq \Pro^2$ (defined over
$\overline{\F_{q^2}}$) such that $\deg(L \cap (E \cup S)) \ge d+2$, or a plane
conic $T$ such that $\deg(T \cap (E \cup S)) \ge 2d+2$.

\begin{enumerate}

\item[(2.i)] Assume the existence of a line $L\subseteq \Pro^2$ such that
$\deg(L \cap (E \cup S)) \ge d+2$. If we have
$h^1(\Pro^2,\mathcal{I}_{\mbox{Res}_L(E \cup S)}(d-1))=0$ then Lemma \ref{c1}
implies $S \subseteq L$ and we may repeat the proof of case (A). The support $S$
consists of $w$ points in $L \cap B$ for a certain $L \in \mathcal{R}(\infty)
\sqcup \mathcal{R}$. Every such a line gives a codeword of $C(d,a)$ of weight
$w$. Now assume $h^1(\Pro^2,\mathcal{I}_{\mbox{Res}_L(E \cup S)}(d-1))>0$. Since
$\deg(\mbox{Res}_L(E \cup S)) \le a+w-(d+2) \le 2(d-1)+1$, Lemma \ref{u00.01}
implies the existence of a line $M \subseteq \Pro^2$ such that $\deg(M \cap
\mbox{Res}_L(E \cup S)) \ge (d-1)+2=d+1$. We easily see that $M$ is defined over
$\F_{q^2}$ and not tangent to $X$ in any point (use Lemma \ref{intrette}). Since
$\mbox{Res}_L(S)=S-(S \cap L)$ we get $L \neq M$.
Observe that $\deg((L \cup M)\cap(E \cup S)) = \deg(L \cap (E \cup S))+\deg(M
\cap \mbox{Res}_L(E\cup S)) \ge 2d+3$. Since neither $L$ or $M$ are tangent to
$X$ we have $\deg (E\cap (L \cup M)) \le 2$, with equality if and only if $L,M
\in \mathcal{R}(\infty)$. In this case we have $w \ge 2d+1$ and it will be
(Lemma \ref{intrette}) $d \le q-1$ or $d=q$ and $\deg (E\cap (L \cup M)) \le 1$.
Since $\deg(\mbox{Res}_{L \cup M}(E \cup S)) \le 3d-1-(2d+3) < d-1$, we have
$h^1(\Pro^2,\mathcal{I}_{\mbox{Res}_{L \cup M}(E \cup S)}(d-2))=0$ and applying
Lemma \ref{c1} with $k=2$ we deduce $S \subseteq L \cup M$.

\item[(2.ii)] Assume that there is no line $L \subseteq \Pro^2$ such that
$\deg(L \cap (E \cup S)) \ge d+2$. Then there is a plane conic $T$ (not
neccessairly smooth) such that $\deg(T \cap (E \cup S)) \ge 2d+2$. Since
$\deg(\mbox{Res}_T(E \cup S)) \le 3d-1-(2d+2) \le d-1$ we get
$h^1(\Pro^2,\mathcal{I}_{\mbox{Res}_T(E \cup S)}(d-2))=0$. Lemma \ref{c1}
implies $S \subseteq T$. Assume that $T$ is reducible, say $T=L \cup M$. Since,
by assumption, $\deg(L\cap (E \cup S)) \le d+1$ and $\deg(M\cap (E \cup S)) \le
d+1$ we have $L \neq M$. Since $2d+2= \deg((L \cup M)\cap (E \cup S))=\deg(L
\cap (E \cup S))+\deg(M \cap \mbox{Res}_L(E \cup S))$ we get (by assumption)
$\deg(L \cap (E\cup S))=\deg(M \cap (E\cup S))=d+1$ and $L \cap M \cap
S=\emptyset$. Moreover, if $P_\infty$ appears in $L \cap M$ then $a \ge 2$.
Lemma \ref{intrette} implies that neither $L$ or $M$ can be tangent to $X$ at
any point. Since we assumed $a<d$ then we are done by Lemma \ref{e7}. Now assume
that $T$ is smooth. Since we proved that $S \subseteq T$, Lemma \ref{intcurve}
gives $w=\deg(T \cap S) \ge 2d+2-\min\{ 2,a \}$.
\end{enumerate}
\end{enumerate}

The proof is concluded. \end{proof}

\section*{Conclusion}
The paper describes Hermitian one-point codes from a purely geometric point of
view and provides a geometric interpretation of the dual minimum distance of
such codes. The supports of the dual minimum-weight codewords are geometrically
characterized, leading to precise formulas for their number.
Possible improvements
of the dual codes of Hermitian one-point codes are easily controlled by means of
the geometric setup here presented.
The well-known geometry of tangent lines to the Hermitian curve is applied to
study also some small-weight codewords and their supports.

\end{document}